\documentclass[11pt,leqno,a4]{amsart}
\usepackage{amssymb,amsmath,amscd,graphicx,amsthm,mathrsfs}
\usepackage[pdftex,bookmarks,colorlinks,breaklinks]{hyperref}
\hypersetup{linkcolor=blue,citecolor=red}
\usepackage{float}
\restylefloat{table}
\usepackage{multicol}
\usepackage[numbers]{natbib}

\newtheorem{theorem}{Theorem}

\newtheorem{lemma}{Lemma}
\newtheorem{proposition}{Proposition}

\theoremstyle{definition}
\newtheorem{remark}{Remark}

\usepackage[top=2cm,bottom=2cm,right=2cm,left=2cm]{geometry}
\newcommand\Tstrut{\rule{0pt}{2.5ex}}         % = `top' strut
\newcommand\Bstrut{\rule[-1ex]{0pt}{0pt}}   % = `bottom' strut
\newcommand{\1}{\bf 1}
\newcommand{\m}{m_{\sf f}}
\newcommand{\Q}{\mathbb{Q}}
\newcommand{\Sp}{\mathrm{Sp}}% Sp
\newcommand{\G}{\mathbf{G}_{ad}}% Chevalley groups
\newcommand{\GL}{\mathrm{GL}}% GL
\newcommand{\SL}{\mathrm{SL}}% SL
\newcommand{\PSL}{\mathrm{PSL}}%PSL
\newcommand{\F}{\mathbb{F}_p}% Finite fields F_p
\newcommand{\CG}{{Z}}% Center of any groups.
\newcommand{\C}{\mathbb{C}}% complex numbers
\newcommand{\R}{\mathbb{R}}% Real number
\newcommand{\ZZ}{\mathbb{Z}}% integers
\newcommand{\g}{\mathfrak{g}}% Simple Lie algebra
\newcommand{\h}{\mathfrak{h}}% Cartan subalgebra
\newcommand{\hh}{\mathrm{ht}}% Height function
\newcommand{\hei}{\mathfrak{u}}% Heisenberg subalgebra
\newcommand{\abel}{\mathfrak{a}}% Maximal abelain subalgebra
\DeclareMathOperator{\ad}{ad}% ad
\DeclareMathOperator{\Tr}{Tr}% Trace
\DeclareMathOperator{\Cay}{Cay}% Cayley graph
\DeclareMathOperator{\car}{char}%char
\newcommand{\bx}[1]{\bar{x}_{#1}}
\newcommand{\FF}[1]{\mathbb{F}_{#1}}
\newcommand{\OO}[1]{\mathcal{O}/\mathfrak{p}^{#1}}
\newcommand{\Z}[1]{\mathbb{Z}/#1\mathbb{Z}}% Z/p^nZ
\newcommand{\la}[2]{\langle#1,#2\rangle}% <a,b>
\newcommand{\res}[2]{\left.#1\right|_#2}% restriction of a function to a set

\title[Faithful representations of Chevalley groups]{Faithful representations of Chevalley groups over quotient rings of non-Archimedean local fields}
\author[M. Bardestani]{Mohammad Bardestani}
\address{Mohammad Bardestani, Department of Mathematics and Statistics, University of Ottawa, 585 King Edward, Ottawa, ON K1N
6N5, Canada.}
\email{mbardest@uottawa.ca}
\author[C. Karimianpour]{Camelia Karimianpour}
\address{Camelia Karimianpour, Department of Mathematics and Statistics, University of Ottawa, 585 King Edward, Ottawa, ON K1N
6N5, Canada.}
\email{ckari099@uottawa.ca}
\author[K. Mallahi-Karai]{Keivan Mallahi-Karai}
\address{Keivan Mallahi-Karai, Jacobs University Bremen, Campus Ring I, 28759 Bremen, Germany.}
\email{k.mallahikarai@jacobs-university.de }
\author[H. Salmasian]{Hadi Salmasian}
\address{Hadi Salmasian, Department of Mathematics and Statistics, University of Ottawa, 585 King Edward, Ottawa, ON K1N
6N5, Canada.}
\email{hadi.salmasian@uottawa.ca}
\begin{document}
\begin{abstract}
Let $F$ be a non-Archimedean local field  with the ring of integers $\mathcal{O}$ and the prime ideal $\mathfrak{p}$ and let $G=\G\left(\OO{n}\right)$ be the adjoint Chevalley group. Let $\m(G)$ denote the smallest possible dimension of a faithful representation of $G$. Using the Stone-von Neumann theorem, we determine a lower bound for $\m(G)$ which is asymptotically the same as the results of Landazuri, Seitz and Zalesskii for  split Chevalley groups over $\FF{q}$. Our result yields a conceptual explanation of the exponents that appear in the aforementioned results  
\end{abstract}
\subjclass{Primary 20C33; Secondary 20G35.}
\keywords{Chevalley groups; faithful representation; Heisenberg subgroups; Local fields; Stone-von Neumann theorem.}
\maketitle
%==================================================
\section{Introduction}
For a finite group $G$, let $\mathrm{Rep}_{\sf f}(G)$ denote the set of all finite dimensional faithful representations of $G$ over  complex vector spaces, and set
\[\m(G):=\min\{d_\rho\,:\,\rho\in\mathrm{Rep}_{\sf f}(G)\},
\]
where $d_\rho$ denotes the dimension (also called the degree) of $\rho$.
Lower bounds on $\m(G)$ can be found in group theory literature as old as the work of Frobenius~\cite{Frob-SL}. Indeed, by constructing the character table of $\PSL_{2}(\F)$, Frobenius  showed that
$$
\m\left(\PSL_{2}(\F)\right)\geq\frac{p-1}{2}\text{ for every prime }p\geq 5.
$$
Apart from its intrinsic interest, the Frobenius bound has applications in many questions in number theory and additive combinatorics. To name a few, Sarnak and Xue~\cite{Sarnak-Xue} were the first to use this bound to obtain a lower bound for the smallest non-trivial eigenvalue of the Laplace-Beltrami operator on the hyperbolic space. This idea was subsequently used by Bourgain and Gamburd~\cite{Bourgain-Gamburd} to answer the $1$-$2$-$3$ question of Lubotzky on the uniform expansion bounds for the Cayley graphs of $\mathrm{SL}_2\left(\F\right)$.

The Frobenius bound has been generalized by Landazuri, Seitz and Zalesskii~\cite{Landazuri-Seitz,Zalesskii} to other families of finite simple groups of Lie type. These bounds play an essential role in the theory of expander graphs and approximate groups~\cite{Breuillard,Lubo}. 
Some of the finite simple groups of Lie type are canonically obtained by reduction mod $\mathfrak p$ of the group $\mathbf G(\mathcal O)$ of $\mathcal O$-points of a Chevalley group, where $\mathcal O$ is the ring of integers of a local field. 
However, despite interesting applications (see below), little work has been done to extend the aforementioned bounds to reduction mod $\mathfrak p^n$ of $\mathbf G(\mathcal O)$. Bourgain and Gamburd~\cite{Bourgain-GamburdI} considered this problem for $\SL_2\left(\Z{p^n}\right)$ in order to show that, for any sufficiently large prime $p$ and any
 symmetric set $S$ generating a Zariski-dense subgroup of $\SL_2(\ZZ)$,
the family of Cayley graphs $\left\{\Cay\left(\SL_2\left(\Z{p^n}\right),\pi_{p^n}(S)\right)\right\}_{n\geq 1}$ is an expander family
(here $\pi_{p^n}$ is the reduction map modulo $p^n$). Indeed, they proved
$$
\m\left(\SL_2\left(\Z{p^n}\right)\right)\geq \frac{p^{n-2}(p^2-1)}{2}
\text{\ \  for }n\geq 2.
$$
Motivated by the works of Bourgain and Gamburd mentioned above,
the first and third authors of this paper studied $\m\left(\SL_k\left(\Z{p^n}\right)\right)$ and $\m\left(\Sp_{2k}\left(\Z{p^n}\right)\right)$~\cite{Bar-Mal}. The same problem for $\m\left(\SL_k\left(\Z{p^n}\right)\right)$ has  been considered by de Saxc{\'e}~\cite{Saxce}.

Let $F$ be a non-Archimedean local field with the ring of integers $\mathcal{O}$ and the prime ideal $\mathfrak{p}$. The order of residue field $\mathcal{O}/\mathfrak{p}$ is denoted by $q=p^l$ where $p$ is a prime number. Our aim in this paper is to obtain a bound for the minimal dimension of all faithful complex representations of adjoint Chevalley groups over the ring $\OO{n}$ where $n$ is a positive integer. Indeed for the Chevalley group $\G(\OO{n})$ associated  to a simple Lie algebra $\g$ we will obtain the following bound
\begin{equation}\label{asy}
\m(\G(\OO{n})) \geq C q^{\frac{n(r+1)}{2}},
\end{equation}
 where $r$ is the dimension of the nilpotent radical of the Heisenberg parabolic subalgebra of $\g$, and $C>0$ is an absolute constant (independent of $q$ and $r$). 

 We remark that all adjoint Chevalley groups over $\mathcal{O}/\mathfrak{p}$ (except for a few cases) are simple groups, and so all of their non-trivial representations are faithful. Therefore, our Theorem 
\ref{Main-Theorem} below 
 for $n=1$ yields lower bounds for non-trivial representations of $\G(\FF{q})$, which,  are asymptotically the same as the  results of~\cite{Landazuri-Seitz,Zalesskii} for split Chevalley groups over $\FF{q}$. By being asymptotically the same we mean that the exponents that appear in~\eqref{asy} are the same as those in the work Landazuri, Seitz and Zalesskii. For a corrigendum to~\cite{Landazuri-Seitz} the cases  ${\sf F}_4(q)$, $q$ odd, and $^2{\sf E}_6(q)$, we refer the reader to~\cite{Zalesskii}. 
 
Let us briefly sketch the idea of this paper. For a given simple Lie algebra $\g$, following Gross and Wallach's idea~\cite{Gross-Wallach}, we consider its Heisenberg parabolic subalgebra whose nilpotent radical is a two step nilpotent subalgeba. This nilpotent subalgebra gives rise to a two step nilpotent subgroup of the adjoint Chevalley group associated to $\g$ (this is the general $(2n+1)$-dimensional Heisenberg group for some $n$). The irreducible representations of a Heisenberg group are classified by their central characters via the Stone-von Neumann theorem. Given a faithful representation $\rho$ of $\G(\OO{n})$, we consider its restriction to the aforementioned Heisenberg subgroup and we find the polarizing subgroup of the generic character of the center. Our bound then is obtained by orbit counting of the action of a certain subgroup on an irreducible component of the representation $\rho$. Our main theorem is the following:
\begin{theorem}\label{Main-Theorem} Let $F$ be a non-Archimedean local field with the ring of integers $\mathcal{O}$, prime ideal $\mathfrak{p}$, and residue field $\mathcal{O}/\mathfrak p\cong\FF{q}$, where $q=p^l$ for a prime number $p$. Let $\g$ be a finite dimensional complex simple Lie algebra with root system $\Phi$. Let 
$$\G\left(\OO{n}\right):=\G\left(\OO{n},\Phi\right),$$
be the adjoint Chevalley group associated to $\g$. Then
$
\m\left(\G\left(\OO{n}\right)\right)\geq h_{\mathrm f}(\Phi,q,n)
$,
where $h_{\mathrm f}(\Phi,q,n)$ is given in the following table
\[
\begin{array}{||l c c c c l||}
\hline
\Phi &&&& &{h_{\mathrm{f}}(\Phi,q,n) }\Tstrut\Bstrut\\
\hline
{{\sf A}_1} &&{p\geq 3}&&{} &{\frac{1}{2}\left(q^n-q^{n-1}\right)}\Tstrut\Bstrut\\
\hline
{{\sf A}_m} &&{m\geq 2,\, p\geq 3}&& &{\left(q^n-q^{n-1}\right)q^{(m-1)n}}\Tstrut\Bstrut\\
\hline
{{\sf B}_m} &&{m\geq 3,\, p\geq 3}&& &{\left(q^n-q^{n-1}\right)q^{(2m-3)n}}\Tstrut\Bstrut\\
\hline
{{\sf C}_m} &&{m\geq 2,\, p\geq 3}&& &{\frac{1}{2}\left(q^n-q^{n-1}\right)q^{(m-1)n}}\Tstrut\Bstrut\\
\hline
{{\sf D}_m} &&{m\geq 4,\, p\geq 3}&& &{\left(q^n-q^{n-1}\right)q^{(2m-4)n}}\Tstrut\Bstrut\\
\hline
{{\sf G}_2}    &&{p\geq 5}&&           &{\left(q^n-q^{n-1}\right)q^{2n}}\Tstrut\Bstrut\\
\hline
{{\sf F}_4}    &&{p\geq 3}&&           &\left(q^n-q^{n-1}\right)q^{7n}\Tstrut\Bstrut\\
\hline
{{\sf E}_6}    &&{p\geq 3}&&           &{\left(q^n-q^{n-1}\right)q^{10n}}\Tstrut\Bstrut\\
\hline
{{\sf E}_7}    &&{p\geq 3}&&           &{\left(q^n-q^{n-1}\right)q^{16n}}\Tstrut\Bstrut\\
\hline
{{\sf E}_8}    &&{p\geq 3}&&           &{\left(q^n-q^{n-1}\right)q^{28n}}\Tstrut\Bstrut\\
\hline
\end{array}
\]
\end{theorem}
\medskip

Unfortunately our proof of this theorem fails when $\car(\mathcal{O}/\mathfrak{p})=2$. We refer the reader to Remark~\ref{remaark-fail}, for further details. 
 
\begin{remark}
Ree~\cite{Ree} (see also~\cite[Theorem 11.3.2]{Carter}) proved that the groups $\G(\mathbb{F}_q)$ are indeed what one would expect to obtain, namely, $\mathrm{PSL}_{m}(\mathbb{F}_q)$ if $\g$ is of type ${\sf A}_{m-1}$;  $\mathrm{PSp}_{2m}(\mathbb{F}_q)$ if $\g$ is of type ${\sf C}_m$; $\mathrm{P\Omega_{2m}}(\mathbb{F}_q)$ if $\g$ is of type ${\sf D}_m$ and $\mathrm{P\Omega}_{2m+1}(\mathbb{F}_q)$ if $\g$ is of type ${\sf B}_m$ and $q\geq 3$. Moreover Chevalley proved that the group $\G(\mathbb{F}_q)$ is simple except for ${\sf A}_1(2)$, ${\sf A}_1(3)$, ${\sf B}_2(2)$ and ${\sf G}_2(2)$ (see~\cite[Theorem 11.1.2]{Carter}).
\end{remark}

\begin{remark}
The above theorem is also valid for \emph{simply connected} Chevalley groups with a similar proof. The sharpness of Landazuri--Seitz bounds (corresponding to the case $n=1$) has been the source of several investigations~\cite{Guralnick, Lubeck, Tiep}. It would be interesting to consider sharpness of our bounds. 
For type $A_{\ell}$, the permutation representation of $\mathrm{PSL}_{\ell+1}(\mathcal O/\mathfrak p^n)$ on the projection space $\mathbb{P}_{\mathcal O/\mathfrak p^n}^{\ell}$ is faithful of dimension $O(q^{\ell n})$, which is consistent with our lower bound. However, we do not know about sharpness of our bounds in general for other groups.
\end{remark}
\begin{remark}
Let us point out that the idea of restriction to nilpotent subgroups was also used in~\cite{Landazuri-Seitz}. Nevertheless, the arguments in~\cite{Landazuri-Seitz} are long and case by case. One of our main goals in writing this paper is to  give  a uniform argument based on the idea of Heisenberg parabolic subalgebras
to obtain lower bounds which,
in the special case of $\mathcal O/\mathfrak p$,
are asymptotically the same as those 
given in ~\cite{Landazuri-Seitz}. (See the discussion after
\eqref{asy} for a precise meaning.)
Such bounds are enough for the existing applications. Another important 
technical detail that has been worked out in our paper is 
to verify that many facts about  Chevalley groups over fields remain valid for Chevalley groups over rings of our interest (see Section~\ref{sectheheis}).
\end{remark}
%=======================================================
\section{Notations and preliminaries}
In this section we set some notation which will be used throughout this paper. We also recall some basic facts about local fields that can be found in~\cite{Neukirch,Serre-Local}.

If $X$ is any set, $f$ any function on $X$, and $Y\subseteq X$ any subset, then $\res{f}{Y}$ is the restriction of $f$ to $Y$. $|X|$ is the cardinality of a finite set $X$. We will use the shorthand ${\bf e}(x) :=\exp(2\pi ix)$. For a given group $G$, its identity element is denoted by $\1$. Moreover $\car(F)$ is the characteristic of a given field.

 By the well-known classification of local fields, any non-Archimedean local field is isomorphic to a finite extension of $\Q_p$ ($p$ is a prime number) or is isomorphic to the field of formal Laurent series $\FF{q}((T))$ over a finite field with $q=p^l$ elements. For a non-Archimedean local field $F$ with the discrete valuation $\nu$, we will denote its ring of integers and its unique prime ideal by $\mathcal{O}$ and $\mathfrak{p}$, respectively. We will also fix a uniformizer $\varpi\in\mathfrak p$. 
For any integer $m\in\mathbb{Z}$, we write \[
\mathfrak{p}^m:=\{x\in F: \nu(x)\geq m\}.
\]
Then 
$\mathfrak{p}^{m}/\mathfrak{p}^{m+n}
\cong\mathcal{O}/\mathfrak{p}^n$ as
additive groups, for every $m,n\in\mathbb Z$ with $n>0$. Let $n$ be a positive integer. Our goal in this section is to describe all additive characters of the finite local rings $\OO{n}$ using the ring structure. 
 
From now on, if $\car(F)=0$ we set $E=\mathbb Q_p$ and if 
$\car(F)=p>0$ we set $E=F$. Now we define
\[
\Tr:=\Tr_{F/E}:F\to E,
\] the trace map of $F$ over $E$.
The {\it Dedekind's complementary module}, (or {\it inverse different}) is defined by 
$$
\mathcal{O}^*:=\{x\in F: \nu(\Tr(sx))\geq 0\,\, \text{for all $s\in\mathcal{O}$}\}.
$$
One can show that $\mathcal{O}^*$ is a fractional ideal of $F$ and hence for some $\ell\geq 0$ we have $\mathcal{O}^*=\varpi_F^{-\ell}\mathcal{O}=\mathfrak{p}^{-\ell}$. Throughout this paper $\ell$ designates this exponent. 
Note that $\ell=0$ when $\car(F)>0$. 
 
We will now construct additive characters $\psi: E \to \mathbb C^*$, where $E$ is as above. 
Let us first consider the zero characteristic case. For every $x\in\mathbb{Q}_p$, let $n_x$ be the smallest non-negative integer such that $p^{n_x}x\in\mathbb{Z}_p$. Let $r_x\in\mathbb{Z}$ be such that $r_x\equiv p^{n_x}x \pmod{p^{n_x}}$. It is easy to see that the following map (known as the Tate character)
\begin{equation}\label{additive-char}
\psi: \Q_p\to \C^*,\qquad x\mapsto {\bf e}(r_x/p^{n_x}),
\end{equation}
is a non-trivial additive character of $\Q_p$ with the kernel $\ZZ_p$.

Similarly for $\FF{q}((T))$, we have $\mathcal{O}=\FF{q}[[T]]$ and $\varpi=T$. We now set
\begin{equation}\label{Laur}
\psi: \FF{q}((T))\to \C^*,\qquad \sum_{i\geq N} a_i T^i \to {\bf e}\left(\Tr_{\mathbb F_q/\mathbb F_p}(a_{-1})/p\right).
\end{equation} 
Notice that the trace map from $\FF{q}$ to $\F$ is surjective. Hence, $\res{\psi}{{\mathfrak{\mathcal O}}}=1$
but $\res{\psi}{{\mathfrak{\mathfrak p}^{-1}}}\neq 1$
(sometimes we say that the conductor of $\psi$ is $\mathcal{O}=\FF{q}[[T]]$). 

 \begin{lemma}\label{Character-local} Let $F$ be a local field with the ring of integers $\mathcal{O}$ and prime ideal $\mathfrak{p}$. All additive characters of the ring $\OO{n}$ are given by
\begin{equation*}
\psi_{\bar{b}}: \OO{n}\to \C^*,\qquad
x+\mathfrak{p}^n\mapsto \psi(\Tr(bx)),
\end{equation*}
where $\bar{b}=b+\mathfrak{p}^{-\ell}\in\mathfrak{p}^{-(n+\ell)}/\mathfrak{p}^{-\ell}$. 
\end{lemma}
\begin{proof}
First assume that $\car(F)=0$, that is, $F$ is a $p$-adic field. Let $\bar{b}_1=b_1+\mathfrak{p}^{-\ell}$ and $\bar{b}_2=b_2+\mathfrak{p}^{-\ell}$ be distinct elements and assume that $\psi_{\bar{b}_1}=\psi_{\bar{b}_2}$. Then for all $x\in \mathcal{O}$ we have $\psi(\Tr((b_1-b_2)x))=1$, which implies that $\Tr((b_1-b_2)x)\in \ZZ_p$. Thus $b_1-b_2\in \mathfrak{p}^{-\ell}$, which is a contradiction. This construction provides exactly $|\mathfrak{p}^{-(n+\ell)}/\mathfrak{p}^{-\ell}|$ distinct additive characters.  Since $|\OO{n}|=|\mathfrak{p}^{-(n+\ell)}/\mathfrak{p}^{-\ell}|$, we are done.

Next assume that $\car(F)>0$. It is clear that the map $\psi_{\bar{b}}$ is well defined. Now suppose for some $b\in\mathfrak{p}^{-n}$ we have $\psi(bx)=1$ for all $x\in\mathcal{O}$. Hence the fractional ideal $b\mathcal{O}$ is a subset of $\ker(\psi)$. Therefore $b\in\mathcal{O}$ since the conductor of $\psi$ is $\mathcal{O}$. This construction provides exactly $|\mathfrak{p}^{-n}/\mathcal{O}|$ distinct additive characters.  Since $|\OO{n}|=|\mathfrak{p}^{-n}/\mathcal{O}|$, we are done.
\end{proof}
%=========================================
\section{The Stone-von Neumann theorem}
In this section, we state a version of Stone-von Neumann theorem that suits our purposes in this paper. The Stone-von Neumann theorem holds in a broader setting~\cite{Howe,Mcnamara,Marie-France}. However, we only present it in the finite group case, which is needed in this paper. 

 Let $U$ be a finite two step nilpotent group, and let $Z(U)$ denote its center. If $A$ is any subgroup of $U$
containing $Z(U)$, we will denote $\bar{A}:= A/Z(U)$. 
Let $\chi:Z(U)\to \mathbb C^*$ be a one-dimensional representation of $Z(U)$. We define a pairing  
$$
U/Z(U)\times U/Z(U)\to\mathbb C^*\ ,\ \left\langle xZ(U),yZ(U)\right\rangle_\chi:=\chi\left([x,y]\right).
$$
We call $\chi$  a {\it generic character of $Z(U)$} if the above pairing is non-degenerate, in the sense that for every $x\in U$, 
if $\left\langle xZ(U),yZ(U)\right\rangle_\chi=1$ for every $y\in U$, then $x\in Z(U)$.
Assuming $\chi$ is generic character of $Z(U)$, we say that a subgroup $Z(U)\leq A\leq U$ is {\it isotropic} if $\bar{A}\subseteq \bar{A}^\perp$, where
\[
\bar{A}^\perp:=\left\{xZ(U)\,:\,
\left\langle xZ(U),yZ(U)\right\rangle_\chi=1\text{ for all }y\in A
\right\}
.\] 
We say that $A$ is {\it polarizing} if $\bar{A} =\bar{A}^\perp$.
For the next theorem we refer the reader to~\cite[$\S$4.1]{Bump}. 
\begin{theorem}[Stone-von~Neumann theorem]\label{SV} Let $U$ be a finite two step nilpotent group, and let $\chi$ be a generic character of $Z(U)$. Then
there exists a unique isomorphism class of irreducible representations of $U$ with central character $\chi$. Such a representation may be constructed as follows:
Let $A$ be any polarizing subgroup of $U$, and let $\tilde{\chi}$ be any extension of $\chi$ to $A$.
Then the representation $\mathrm{Ind}_A^U(\tilde{\chi})$ is of this class.
\end{theorem}
%============================================================
\section{The Heisenberg parabolic subalgebra}
\label{sectheheis}
This section is devoted to a rapid review of some basic facts in the theory of simple Lie algebras. We closely follow Gross and Wallach's paper~\cite{Gross-Wallach}, Sections 1 and 2 (see also~\cite[$\S$3]{Hadi}). 
Let $\g$ be a complex finite dimensional simple Lie algebra. Fix a Cartan subalgebra $\h$ of $\g$. Let $\Phi\subseteq \h^*$ be the root system of $\g$ with respect to $\h$. Then, we have the Cartan decomposition
\begin{equation}\label{Cartan-decomp}
\g=\h\oplus\bigoplus_{\alpha\in \Phi}\g_\alpha,
\end{equation}
where $\g_\alpha=\{x\in\g :\,[H,x]=\alpha(H)x,\,\forall  H\in\h\}$. 
Let $E=\mathrm{Span}_{\R}\{\alpha\:|\:\alpha\in{\Phi}\}$. Note that $E$ is equipped with a symmetric positive definite inner product $(\,,\,)$ obtained from the Killing form of $\g$ via the isomorphism between $\h$ and $\h^*$. For $\alpha,\beta\in \Phi$, set $\langle\alpha,\beta\rangle=2(\alpha,\beta)/(\beta,\beta)$.
Let $\beta\neq \pm\alpha$ be two independent roots.  Assume that $\|\beta\|\geq \|\alpha\|$. 
Then the values of $\la{\alpha}{\beta}$ and $\la{\beta}{\alpha}$ are given by 
Table~\ref{Root structure}
(see~\cite[Table 1, $\S$9.4]{Humphreys}).
\begin{table}[H]
\caption{Root structure}\label{Root structure}
\centering
\resizebox{5.2cm}{!}{
\begin{tabular}{r r c r}
\hline \Tstrut\Bstrut
  $\la{\alpha}{\beta}$ & $\la{\beta}{\alpha}$ & {} & $\left(\|\beta\|/\|\alpha\|\right)^2$ \\
  \hline \Tstrut\Bstrut
   {$0$} & {$0$} & {} & \text{undetermined}  \\
   {$1$} & {$1$} & {} & {$1$}  \\
   {$-1$} & {$-1$} & {} & {$1$}  \\
   {$1$} & {$2$} & {} & {$2$}  \\
   {$-1$} & {$-2$} & {} & {$2$}  \\
   {$1$} & {$3$} & {} & {$3$}  \\
   {$-1$} & {$-3$} & {} & {$3$}  \\
 \end{tabular}}
 \end{table}
Let $\Delta$ be a base of $\Phi$. Let $\Phi^+\subseteq \Phi$ the set of positive roots with respect to $\Delta$, and let $\tilde{\beta}$ be the highest root. It is known that $\tilde\beta$ is a long root and $m_\alpha\geq n_\alpha$, where $\tilde{\beta}=\sum_{\alpha\in\Delta}m_{\alpha}\alpha$ and $\gamma=\sum_{\alpha\in\Delta}n_\alpha\alpha$ is any $\gamma\in\Phi$. Given the above notation, we define the {\it Heisenberg parabolic subalgebra} $\mathfrak{q}=\mathfrak{l}\oplus\hei$. The Levi subalgebra and the nilpotent radical of $\mathfrak{q}$ are:
$$
\mathfrak{l}=\h\oplus\bigoplus_{\substack{\alpha\in\Phi\\\langle \alpha,\tilde{\beta}\rangle= 0}}\g_\alpha,\qquad \text{and}\qquad \hei=\bigoplus_{\substack{\alpha\in\Phi\\\langle \alpha,\tilde{\beta}\rangle> 0}}\g_\alpha.
$$
\begin{lemma}\label{Lemma-rootproperty} The inequality $\langle\alpha,\tilde{\beta}\rangle> 0$ implies that $\alpha\in\Phi^+$, and either $\alpha=\tilde{\beta}$ or $\langle\alpha,\tilde{\beta}\rangle=1$. Moreover, if $\langle\alpha,\tilde{\beta}\rangle=1$, then 
$\tilde{\beta}-\alpha\in\Phi^+$ and 
$\langle\tilde{\beta}-\alpha,\tilde{\beta}\rangle=1$.
\end{lemma}
\begin{proof}
 If $\langle\alpha,\tilde{\beta}\rangle>0$ then $\tilde{\beta}-\alpha\in\Phi$ (see~\cite[Lemma of $\S$9.4]{Humphreys}). This implication, along with the fact that $\tilde{\beta}$ is the highest root, implies $\alpha\in\Phi^+$.

Note that for any $\alpha\in\Phi^+$, $|\langle\alpha,\tilde{\beta}\rangle|\leq |\langle\tilde{\beta},\alpha\rangle|$. Assume $\alpha\neq \tilde{\beta}$. Then by applying Table~\ref{Root structure} and a simple calculation we deduce that $\langle\alpha,\tilde{\beta}\rangle \langle\tilde{\beta},\alpha\rangle\in\{1,2,3\}$. Hence, $\langle\alpha,\tilde{\beta}\rangle>0$ implies $\langle\alpha,\tilde{\beta}\rangle=1$. The last claim in the statement follows from linearity of $\langle\:,\:\rangle$ in the first component.
\end{proof}
 Let $\Sigma^+:=\{\alpha\in \Phi^+: \langle \alpha,\tilde{\beta}\rangle=1 \}$. Lemma~\ref{Lemma-rootproperty} allows us to define a fixed-point free involution of $\Sigma^+$ defined by $\alpha\mapsto \tilde{\beta}-\alpha$. We pick one element from each equivalence class. Therefore, we have the following disjoint decomposition:
\begin{equation}\label{Sigma-decomposition}
\Sigma^+=\{\alpha_i: 1\leq i\leq d\}\cup\{\tilde{\beta}-\alpha_i: 1\leq i\leq d\}.
\end{equation}
Hence, $|\Sigma^+|=2d$, where the value of the integer $d$ is explicitly calculated in Proposition 1.3 of~\cite{Gross-Wallach}. In particular,
Table \ref{dvalue} is given in~\cite{Gross-Wallach}:
\begin{table}[H]
\caption{Values of $d$}\label{dvalue}
\centering
\begin{tabular}{l c r}
\hline
   {$\g$}  & {} & {$d$}  \\
  \hline
   {${\sf A}_m$} & {$m\geq 1$}  & {$m-1$}  \\
   {${\sf B}_m$} & {$m\geq 2$}  & {$2m-3$}  \\
   {${\sf C}_m$} & {$m\geq 2$}  & {$m-1$}  \\
   {${\sf D}_m$} & {$m\geq 3$}  & {$2m-4$}  \\
   {${\sf G}_2$} & {} &  {2}  \\
   {${\sf F}_4$} & {} & {7}  \\
   {${\sf E}_6$} & {} & {10}  \\
   {${\sf E}_7$} & {} & {16}  \\
   {${\sf E}_8$} & {} & {28}
\end{tabular}
\end{table}
\begin{lemma}\label{u-2nilpotent}
The subalgebra $\hei$ is a two-step nilpotent Lie algebra with center $\g_{\tilde{\beta}}$.
\end{lemma}
\begin{proof}
It follows from Lemma~\ref{Lemma-rootproperty} that $[\hei,\hei]\subseteq\g_{\tilde\beta}\subseteq \mathrm{Z}(\hei)$, which implies $\hei$ is a two-step nilpotent Lie algebra. For $\gamma_1,\gamma_2\in \Sigma^+$, notice that $[\g_{\gamma_1},\g_{\gamma_2}]=0$ unless $\gamma_2=\tilde{\beta}-\gamma_1$, and in this case we have $[\g_{\gamma_1},\g_{\tilde{\beta}-\gamma_1}]=\g_{\tilde{\beta}}$. Using these equalities, one can see that $\g_{\tilde{\beta}}=\mathrm{Z}(\hei)$.
\end{proof}
Notice that $\hei=\g_{\tilde{\beta}}\oplus\bigoplus_{\substack{\alpha\in\Sigma^+}}\g_{\alpha}$ is of dimension $2d+1$. Let us choose a $(d+1)$-dimensional maximal abelian subalgebra $\abel$ of $\hei$, defined to be
\begin{equation}\label{cons-a}
\abel=\g_{\tilde{\beta}}\oplus\bigoplus_{i=1}^d\g_{\alpha_i}.
\end{equation}
The maximality can be seen with the help of Lemma~\ref{Lemma-rootproperty}, specifically the fact that $\g_{\tilde{\beta}}=[\g_{\tilde{\beta}-\alpha_i},\g_{\alpha_i}]$. In Section~\ref{pa-sub-section}, we show that this subalgebra produces a polarizing subgroup of a Heisenberg subgroup.
\begin{lemma}\label{root=12} Let $\alpha\in \Phi$ be an arbitrary root. There exists a simple root  $\gamma\in \Delta$ such that $\langle \alpha,\gamma\rangle=\pm 1\, \text{or}\, \pm 2$.
\end{lemma}
\begin{proof}
 The statement is clear if $\alpha\in \pm\Delta$ (since we can set $\gamma=\pm \alpha$) and so we can assume that $\alpha\in\Phi\setminus \pm\Delta$. In this case
there exists a root $\gamma\in\Delta$ such that $\la{\alpha}{\gamma}\neq 0$. For root systems other than ${\sf G}_2$, the lemma follows from  Table~\ref{Root structure}. For 
${\sf G}_2$, the lemma can be verified by a direct examination of the roots.
\end{proof}
Set
\begin{equation}
F(\Phi):=\min\left\{\langle\tilde{\beta},\alpha\rangle> 0: \alpha\in\Phi\right\}.
\end{equation}
Obviously $F(\Phi)\leq 2$. For the root systems ${\sf A}_m, {\sf D}_m, {\sf E}_6, {\sf E}_7$ and ${\sf E}_8$ have only one root length and so a similar argument as above shows that $F(\Phi)=1$ unless $\Phi={\sf A}_1$ which in this case we have $F({\sf A}_1)=2$. For ${\sf B}_m, {\sf F}_4$, and ${\sf G}_2$, we observe that these root systems have non-perpendicular long roots and so for these root systems we also have $F(\Phi)=1$.

We show that $F({\sf C}_m)=2$. If $\langle \tilde{\beta},\alpha\rangle=1$ then $\alpha$ is a long root, but in ${\sf C}_m$ all non-proportional distinct long roots are perpendicular. Hence $F({\sf C}_m)=2$.  Therefore we have
\begin{equation}\label{F(Phi)}
F(\Phi)=\left\{
          \begin{array}{ll}
            1, & \Phi\neq {\sf A}_1;\ {\sf C}_m,\ m\geq 2\\
            2, &  \Phi= {\sf A}_1;\ {\sf C}_m,\ m\geq 2.
          \end{array}
        \right.
\end{equation}
%======================================
\section{Heisenberg subgroups of adjoint Chevalley groups}\label{pa-sub-section}
In this section we review the construction of {\it elementary adjoint Chevalley groups} and we define Heisenberg subgroups of Chevalley groups which are 
obtained by exponentiating the nilpotent radical of the Heisenberg parabolic  subalgebra $\mathfrak{q}$ defined in the previous section. Moreover, we verify that the construction of Chevalley groups over fields given in~\cite{Carter,Steinberg} can be extended to elementary Chevalley groups defined over $\OO{n}$. One way to approach this is to use the language of group schemes~\cite{Borel-Group-Schme}. However, in this paper we consider the explicit construction of Chevalley groups using Chevalley bases. The theory of elementary Chevalley groups over rings has also been presented in detail in~\cite{Vavilov-Plotkin}.

As before $F$ is a non-Archimedean local field with the ring of integers $\mathcal{O}$, the prime ideal $\mathfrak{p}$ and the residue field $\FF{q}$, $q=p^l$. Here we assume that $p\geq 3$ and we set $R=\OO{n}$, $n\geq 1$. We will use the standard notation, which can be found in~\cite{Carter,Serre,Steinberg}. Let
$$
\{H_\alpha:\,\, \alpha\in\Delta\}\cup\{e_\alpha:\,\, \alpha\in \Phi\},
$$
be a {\it Chevalley basis}, with respect to our choice of base $\Delta$. Let $\g_\ZZ\subseteq \g$ be the free $\ZZ$-module generated by the Chevalley basis. One can show that $\g_\ZZ$ is indeed a Lie algebra over $\ZZ$. For any $\alpha\in\Phi$ and $\xi\in\mathbb{C}$, 
$\ad_{\xi e_\alpha^{}}=\xi\ad_{e_\alpha}$ is a nilpotent derivation of $\g$. Hence, the exponential map
$$x_\alpha(\xi):=\exp(\xi\ad_{e_\alpha}),$$
 is a Lie algebra automorphism of $\g$. Moreover, the entries of the matrix of $x_\alpha(\xi)$, with respect to the Chevalley basis, are of the form $a\xi^i$, where $a\in\ZZ$ and $i$ is a non-negative integer. Let us denote this matrix by $A_\alpha(\xi)$. Consider the $R$-Lie algebra $\g_{R}:=\g_\ZZ\otimes_{\ZZ}R$ with the Chevalley basis
$$
\{H_\alpha=H_\alpha\otimes 1 :\,\, \alpha\in\Delta\}\cup\{\,\, e_\alpha=e_\alpha\otimes 1 :\,\, \alpha\in \Phi\}.
$$
For every $t\in R$, we obtain a new matrix $\bar{A}_{\alpha}(t)$ from $A_\alpha(\xi)$, by replacing the entries $a\xi^i$ by $\bar{a}t^i$, where $\bar{a}$ is $a$ reduced modulo $\mathfrak{p}^n$. The linear transformation $\bar{x}_\alpha(t)$ associated with the matrix $\bar{A}_{\alpha}(t)$ is a Lie algebra automorphism of $\g_R$. The subgroup of the automorphism group of
$\g_R$, generated by transformations $\bar{x}_\alpha(t)$ for each $\alpha\in\Phi$ and $t\in R$, is called the {\it elementary adjoint Chevalley group}. We denote it by $\G\left(R\right):=\G\left(R,\Phi\right)$. Let $\alpha\in \Phi$ be an arbitrary root. The {\it one-parameter subgroup} $X_\alpha$ of $\G\left(R\right)$ is defined by
$$X_\alpha=\left\langle \bar{x}_\alpha(t): t\in R\right\rangle.$$
\begin{lemma}\label{Xbeta}
The subgroup $X_{\alpha}$ is isomorphic to the additive group of $R$.
\end{lemma}
\begin{proof}
The map $t\rightarrow {\bar{x}_\alpha(t)}$ gives the desired group isomorphism. Note that the injectivity can be seen through the action of $\bar{x}_\alpha(t)$ on the Chevalley basis for the Lie algebra $\g_R$. More precisely, we have (see~\cite{Carter}, $\S$4.4) $\bar{x}_\alpha(t)H_{\gamma}=H_{\gamma}-\langle\alpha,\gamma\rangle te_\alpha$. By  Lemma~\ref{root=12}, if $\bar{x}_\alpha(t)={\bf 1}$, then $t=0$ since $p\geq 3$.
\end{proof}
Let us define the {\it Heisenberg subgroup} $U$ of $\G\left(R\right)$
\begin{equation}\label{Heis-Par}
U=\left\langle\bar{x}_\alpha(t): \la{\alpha}{\tilde{\beta}}\geq 1,\, t\in R\right\rangle.
\end{equation}
 Here the right hand side of~\eqref{Heis-Par} is
 the subgroup of $\mathbf G_{ad}(R)$ generated by the given elements $\bar{x}_\alpha(t)$.
This subgroup is analogous to the nilpotent radical of the Heisenberg parabolic subalgebra. This analogy will be apparent in Proposition~\ref{Prop-3}. From now on, we fix a total ordering $\prec$ of $\Phi$ which is compatible with the height function $\hh$, i.e. $\alpha\prec \beta$ implies $\hh(\alpha)\leq \hh(\beta)$. We recall a theorem due to Chevalley (the proof over $R$ is similar to~\cite[Theorem 5.2.2]{Carter}) that expresses the commutator of two generators of $\G(R)$ as a product of generators. Let $\alpha, \beta\in\Phi$ such that $\alpha\neq\pm\beta$,  and let $t_1,t_2$ be elements of $R$. Let us define the commutator  $[\bar{x}_\alpha(t_2),\bar{x}_\beta(t_1)]:=\bar{x}_\alpha(t_2)^{-1}\bar{x}_\beta(t_1)^{-1}\bar{x}_\alpha(t_2)\bar{x}_\beta(t_1)$.
The Chevalley commutator formula states that
\begin{equation}\label{Che-comm}
[\bar{x}_\alpha(t_2),\bar{x}_\beta(t_1)]=\prod_{i,j>0}\bar{x}_{i\beta+j\alpha}\left(C_{i,j,\beta,\alpha}(-t_1)^it_2^j\right),
\end{equation}
where the product is taken over all pairs of positive integers $i,j$ for which $i\beta+j\alpha$ is a root, and the terms of the product are in increasing order of $i+j$. The constants $C_{i,j,\beta,\alpha}$ are in the set $\{\pm 1,\pm 2,\pm 3\}$.

Next, we point out that  every element of $U$ can be expressed uniquely in the form
\begin{equation}\label{uniq}
\prod_{\langle\alpha,\tilde{\beta}\rangle\geq 1}\bar{x}_{\alpha}(t_\alpha),
\end{equation}
where the product is taken over positive roots $\alpha$, increasing in the chosen total ordering. Indeed, given an element of $U$ in the form of a product of $\bar{x}_\alpha(t)$'s, the desired order can be achieved by performing a rearrangement as follows: if there is  a pair of consecutive terms $\bar{x}_{\alpha}(t_\alpha)\bar{x}_{\beta}(t_\beta)$ with $\beta\prec \alpha$, we swap them by use of~\eqref{Che-comm}:
\begin{equation}\label{Che-com1}
\bar{x}_{\alpha}(t_\alpha)\bar{x}_{\beta}(t_\beta)=\bar{x}_{\beta}(t_\beta)\bar{x}_{\alpha}(t_\alpha)\prod_{i,j>0}\bar{x}_{i\beta+j\alpha}\left(C_{i,j,\beta,\alpha}(-t_\beta)^it_\alpha^j\right).
\end{equation}
In this fashion, $\bar{x}_{\beta}(t_\beta)\bar{x}_{\alpha}(t_\alpha)$ is in the increasing order, and all the extra terms introduced by use of the commutator formula are in the desired order because the total ordering $\prec$ is compatible with the height function. This rearrangement terminates after finitely many iterations. The uniqueness of such an expression of elements in $U$ is proved by an argument similar to the proof of~\cite[Theorem 5.3.3(ii)]{Carter}. The following lemma can be proved easily.

\begin{lemma}\label{p=3,lemma} Let $\Phi$ be a root system different from ${\sf G}_2$. Then for any $\alpha_i$, chosen from the decomposition~\eqref{Sigma-decomposition} and $t\in R$, we have
$$
\bar{x}_{\alpha_i}(1)\bar{x}_{\tilde{\beta}-\alpha_i}(t)=\bar{x}_{\tilde{\beta}-\alpha_i}(t)\bar{x}_{\alpha_i}(1)\bar{x}_{\tilde{\beta}}(Ct),
$$
where $C\in \{\pm 1,\pm 2\}$.
\end{lemma}
\begin{proof}
From~\eqref{Che-com1} we have
$$
\bar{x}_{\alpha_i}(1)\bar{x}_{\tilde{\beta}-\alpha_i}(t)=\bar{x}_{\tilde{\beta}-\alpha_i}(t)\bar{x}_{\alpha_i}(1)\bar{x}_{\tilde{\beta}}(-C_{1,1,\tilde{\beta}-\alpha_i,\alpha_i}t).
$$
But (see~\cite{Carter}, Theorem 5.2.2) $C_{1,1,\tilde{\beta}-\alpha_i,\alpha_i}=\pm(r+1)$,
where
$$
(\tilde{\beta}-\alpha_i)-r\alpha_i,\cdots,(\tilde{\beta}-\alpha_i),\cdots,(\tilde{\beta}-\alpha_i)+s\alpha_i,
$$
is the $\alpha_i$-chain through $(\tilde{\beta}-\alpha_i)$. Since $\tilde{\beta}$ is the highest root, $(\tilde{\beta}-\alpha_i)+s\alpha_i$ is not a root for $s>1$, and therefore $s=1$. Also, it is known that $\langle\tilde{\beta}-\alpha_i,\alpha_i\rangle=r-s$. It follows that
$$
r+1=\langle\tilde{\beta},\alpha_i\rangle.
$$
Notice that $\langle\tilde{\beta},\alpha_i\rangle\in\{1,2\}$, since otherwise $\|\tilde{\beta}\|/\|\alpha_i\|=3$, which is impossible when the root system is different from ${\it G}_2$.
\end{proof}

\begin{lemma}\label{Lemma-Hadi} Let $G$ be a finitely generated group generated  by $g_i$, $1\leq i\leq n$. Let $A\unlhd G$ and assume that $[g_i,g_j]\in A$ for any $1\leq i,j\leq n$. Then $[G,G]\subseteq A$.
\end{lemma}
\begin{proposition}\label{Prop-3}
Let $p\geq 3$ if $G$ is not of type ${\sf G}_2$ and $p\geq 5$ otherwise. Then the group $U$ is two step nilpotent and  $[U,U]=X_{\tilde{\beta}}$.
\end{proposition}
\begin{proof}
With the help of~\eqref{Che-comm} and Lemma~\ref{Lemma-rootproperty}, one can see that the commutators of the generators of $U$ are in $X_{\tilde\beta}$. Hence by applying Lemma~\ref{Lemma-Hadi}  we conclude that the commutator subgroup of $U$ is contained in $X_{\tilde{\beta}}$. Conversely,
by Lemma \ref{p=3,lemma},
 any element in $X_{\tilde{\beta}}$ can be obtained from commuting suitable elements of $X_{\alpha_i}$ and $X_{\tilde{\beta}-\alpha_i}$, for $1\leq i\leq d$. Hence, $[U,U]=X_{\tilde{\beta}}$. On the other hand, the fact that $\tilde\beta$ is the highest root implies that for every $\alpha$ satisfying $\langle\alpha,\tilde\beta\rangle>0$, we have $i\tilde\beta+j\alpha\not\in\Phi$ for all $i,j>0$. Hence, by~\eqref{Che-comm} we have $[X_\alpha,X_{\tilde\beta}]=\bold{1}$ which implies that $X_{\tilde{\beta}}\subseteq\CG(U)$ and hence $U$ is a two step nilpotent subgroup.
\end{proof}

We now  recall that by a theorem of Chevalley (whose proof over $R$ is similar to~\cite[Theorem 6.3.1]{Carter}),
for any root $\alpha$ there exists a surjective homomorphism
\begin{equation}
\phi_\alpha: \SL_2\left(R\right) \longrightarrow \langle X_\alpha,X_{-\alpha}\rangle,
\end{equation}
such that
$$
\phi_\alpha\begin{pmatrix}
1 & t\\
0 & 1
\end{pmatrix}=\bar{x}_\alpha(t),\qquad
\phi_\alpha\begin{pmatrix}
1 & 0\\
t & 1
\end{pmatrix}=\bar{x}_{-\alpha}(t).
$$
For any invertible element $\lambda\in R$, we denote
\begin{equation}\label{hlambda}
h_\alpha(\lambda):=\phi_\alpha \begin{pmatrix}
\lambda & 0\\
0 & \lambda^{-1}
\end{pmatrix}.
\end{equation} 
Let $\alpha,\beta\in\Phi$ be any roots, then one can show that (see~\cite[Chapter 7]{Carter})
\begin{equation}\label{action-h}
h_\alpha(\lambda)\bar{x}_{\beta}(t)h_\alpha(\lambda)^{-1}=\bar{x}_{\beta}\left(\lambda^{\la{\beta}{\alpha}}t\right).
\end{equation}
%===============================================================
\section{Faithful representations and generic characters}

 Let $(\rho, V)$ be a faithful representation of $\G\left(R\right)$ where $R=\OO{n}$. 
Let $\sigma:=\res{\rho}{U}$, be the restriction of $\rho$ to the Heisenberg subgroup $U$, defined in~\eqref{Heis-Par}, and let $(\sigma_i,V_i)$, $1\leq i\leq k$, be the irreducible factors in the decomposition of the $U$-representation $(\sigma,V)$. Then by Schur's lemma for any $z\in\CG(U)$ and $v\in V_i$ we have $\sigma_i(z)v=\chi_i(z)v$, where $\chi_i$ is a one-dimensional representation of $Z(U)$. By Lemma~\ref{Character-local} for each $1\leq i\leq k$ there exists $\bar{b}_i=b_i+\mathfrak{p}^{-\ell}\in\mathfrak{p}^{-(n+\ell)}/\mathfrak{p}^{-\ell}$ such that for any $s\in \mathcal{O}$, 
\begin{equation}\label{char-XB}
\chi_i\left(\bx{\tilde{\beta}}(s+\mathfrak{p}^n)\right)=\psi(\Tr(b_is)).
\end{equation}
 With this observation we prove the following proposition. As before the characteristic of the residue field $\mathcal{O}/\mathfrak{p}$ is $p$. 
 
\begin{proposition}\label{generic-prop} Let $p\geq 3$ when $\Phi\neq {\sf G}_2$ and $p\geq 5$ when $\Phi={\sf G}_2$. Let $\chi_i$, $1\leq i\leq k$ be defined as above, and let $b_i\in \mathfrak{p}^{-(n+\ell)}$ correspond to $\chi_i$ by 
Lemma 
\ref{Character-local}.
Then  
$\nu(b_i)=-(n+\ell)$
for some $1\leq i\leq k$.
In particular, $\chi_i$ 
is a generic character of $Z(U)$. 
\end{proposition}

\begin{proof} 
Suppose that for  each $1\leq i\leq k$ we have $\varpi^{n-1} b_i\in\mathfrak{p}^{-\ell}$. Then $\chi_i(\bx{\tilde{\beta}}(\varpi^{n-1}+\mathfrak{p}^n))=1$. This in particular implies that $\rho(\bx{\tilde{\beta}}(\varpi^{n-1}+\mathfrak{p}^n))=\1$ which is a contradiction since $\rho$ is assumed to be a faithful representation. This proves the existence of $1\leq i\leq k$ such that $\nu(b_i)=-(n+\ell)$.

Next we prove that if $\nu(b_i)=-(n+\ell)$ then $\chi_i$ is a generic character of $Z(U)$. For $u\in U$ let $\chi_i([u,y])=1$ for all $y\in U$. By~\eqref{uniq},
\begin{equation}\label{uu}
u=\prod_{\langle\alpha,\tilde{\beta}\rangle\geq 1}\bar{x}_{\alpha}(s_\alpha+\mathfrak{p}^n)\qquad s_\alpha\in \mathcal{O},
\end{equation}
where the product is taken over positive roots $\alpha$, increasing in the chosen total ordering. We will show that the only term that contributes to~\eqref{uu} is the term that belongs to $X_{\tilde{\beta}}$. It follows that $u\in Z(U)$. Note that, for any $x\in U$, the map $y\mapsto [x,y]$ is a group homomorphism, since $U$ is a two step nilpotent group. 

We remark that for $\alpha,\beta\in \Sigma^+$ we have $[X_\alpha,X_\beta]=\bold{1}$ unless $\beta=\tilde{\beta}-\alpha$. For any $\alpha\neq \tilde{\beta}$ in~\eqref{uu} and arbitrary $s\in \mathcal{O}$, from the Chevalley commutator formula~\eqref{Che-comm} we have 
$$
[u,\bx{\tilde{\beta}-\alpha}(s+\mathfrak{p}^n)]=\bx{\tilde{\beta}}(Cs_\alpha s+\mathfrak{p}^n),
$$ 
where by Lemma~\ref{p=3,lemma}, $C\in \{\pm 1,\pm 2\}$ if $\Phi$ is different from ${\sf G}_2$ and $C\in \{\pm 1,\pm 2,\pm 3\}$ when $\Phi={\sf G}_2$. Since for any $s\in \mathcal{O}$ we have 
\begin{equation}\label{Cb_1s}
1=\chi_i^{}\left([u,\bx{\tilde{\beta}-\alpha}(s+\mathfrak{p}^n)]\right)=\chi_i^{}\left(\bx{\tilde{\beta}}(Cs_\alpha s+\mathfrak{p}^n)\right)=\psi(\Tr(Cb_i s_\alpha s)),
\end{equation} 
then 
\begin{equation}\label{fail}
Cb_is_\alpha\in\mathfrak{p}^{-\ell},
\end{equation}
which implies that $\nu(s_\alpha)\geq n$ since $\nu(b_i)=-(n+\ell)$ (when $\Phi={\sf G}_2$ we must assume $p\geq 5$ since $C$ can be $\pm 3$). Hence $s_\alpha\in\mathfrak{p}^n$. This shows that $u=\bx{\tilde{\beta}}(s_{\tilde{\beta}}+\mathfrak{p}^n)$ for some $s_{\tilde{\beta}}\in \mathcal{O}$ and so $u\in X_{\tilde{\beta}}\subseteq Z(U)$
 which shows that $\chi_i$ is a generic character of $Z(U)$.
\end{proof}
\begin{remark}\label{remaark-fail} When $p=2$, the above argument fails since $C$ can be $\pm 2$ and then we can not deduce from~\eqref{fail} that $\nu(s_\alpha)\geq n$. This is crucial when one wants to apply the Stone-von Neumann theorem. 
\end{remark}
In order to apply the Stone-von Neumann theorem, we then need to find the polarizing subgroup of $U$.  Define 
\begin{equation}\label{polar}
A=\left\langle\bar{x}_{\tilde{\beta}}(t),\,\,\bar{x}_{\alpha_i}(t): 1\leq i\leq d,\, t\in R\right\rangle,
\end{equation}
where $\alpha_i$ are chosen with respect to the decomposition~\eqref{Sigma-decomposition}. We will show that $A$ is a polarizing subgroup of $U$ with respect to the generic character $\chi_i$ defined above.  Notice that for any $\alpha_i$ and $\alpha_j$, $1\leq i,j\leq d$, chosen from the first disjoint component of the decomposition~\eqref{Sigma-decomposition}, neither $\alpha_i+\alpha_j$ nor $\alpha_i+\tilde\beta$ is a root. Hence, the right hand side of~\eqref{Che-comm} is always zero for elements in $A$, and therefore $A$ is an abelian subgroup of $U$ containing $X_{\tilde{\beta}}$.
\begin{proposition}\label{polar-prop} Let $p\geq 3$ when $\Phi\neq {\sf G}_2$ and $p\geq 5$ when $\Phi={\sf G}_2$. 
Let $\chi_1$ be a one-dimensional representation of $Z(U)$ corresponding to 
$b_1\in\mathfrak p^{-(n+\ell)}$ via Lemma \ref{Character-local}. Assume that $\nu(b_1)=-(n+\ell)$. 
Then $A$ is a polarizing subgroup with respect to $\chi_1$. 
\end{proposition}
\begin{proof}
We have shown that $A$ is an abelian subgroup and so $A$ is an isotropic subgroup of $U$. Assume $A$ is not a polarizing subgroup.  Each $u\in U$ has a unique presentation~\eqref{uu}. By the length of $u\in U$ we mean the number of terms in~\eqref{uu}. Let $u\in U\setminus A$ be an element with the shortest length such that
\begin{equation}\label{contra-pola}
\chi_1\left([u,a]\right)=1, \qquad \forall a\in A.
\end{equation}
Let us denote the unique presentation of $u$ as follows 
\begin{equation}\label{uuu}
u=\prod_{\langle\alpha,\tilde{\beta}\rangle\geq 1}\bar{x}_{\alpha}(s_\alpha+\mathfrak{p}^n)\qquad s_\alpha\in \mathcal{O}.
\end{equation}
 We claim that the leftmost term in the product in~\eqref{uuu} cannot belong to either of $X_{\tilde\beta}$ and $X_{\alpha_i}$, $1\leq i\leq d$. That is because otherwise, one can eliminate this  term and obtain another element in $U\setminus A$ with shorter length that satisfies in~\eqref{contra-pola}.
We again remark that for $\alpha,\beta\in \Sigma^+$ we have $[X_\alpha,X_\beta]=\bold{1}$ unless $\beta=\tilde{\beta}-\alpha$. Without loss of generality we can write
$$
u=\bar{x}_{\tilde{\beta}-\alpha_1}(s_{\alpha_1}+\mathfrak{p}^n)u' \qquad s_{\alpha_1}\in \mathcal{O}\setminus \mathfrak{p}^n,
$$
where $\alpha_1$ is taken from the decomposition~\eqref{Sigma-decomposition}. Notice that none of the elements of $X_{\tilde{\beta}-\alpha_1}$ appears in the factorization of $u'$.
Hence for an arbitrary $s\in\mathcal{O}$ we have
\begin{equation*}
\begin{split}
[u,\bar{x}_{\alpha_1}(s+\mathfrak{p}^n)]&=[\bar{x}_{\tilde{\beta}-\alpha_1}(s_{\alpha_1}+\mathfrak{p}^n),\bar{x}_{\alpha_1}(s+\mathfrak{p}^n)][u',\bar{x}_{\alpha_1}(s+\mathfrak{p}^n)]\\
&=[\bar{x}_{\tilde{\beta}-\alpha_1}(s_{\alpha_1}+\mathfrak{p}^n),\bar{x}_{\alpha_1}(s+\mathfrak{p}^n)]=\bx{\tilde{\beta}}(Cs_{\alpha_1} s+\mathfrak{p}^n)
\end{split}
\end{equation*}
where $C\in\{\pm 1,\pm 2,\pm 3\}$. Hence from~\eqref{contra-pola} we deduce that for any $s\in\mathcal{O}$
\begin{equation}\label{Cb_1s_1}
1=\chi_1\left([u,\bx{\alpha_1}(s+\mathfrak{p}^n)]\right)=\psi(\Tr(Cb_1s_{\alpha_1}s)).
\end{equation}
Therefore we should have $Cb_1s_{\alpha_1}\in \mathfrak{p}^{-\ell}$. By Lemma~\ref{p=3,lemma}, for all root systems other than ${\sf G}_2$, the contradiction $s_{\alpha_1}\in \mathfrak{p}^n$ is obtained when $p\geq 3$. However, for the root system ${\sf G}_2$ the further assumption $p\geq 5$ is required in order to obtain a contradiction.
\end{proof}
We now compute the index of $A$ in $U$.
\begin{lemma}\label{index}
Let $d$ be as in Table~\ref{dvalue}. Then $[U:A]=q^{nd}$.
\end{lemma}
\begin{proof}
Recall that $|\Sigma^+|=2d$. From~\eqref{Heis-Par} and uniqueness of the product in ~\eqref{uniq} we deduce that $|U|=q^{(2d+1)n}$ and $|A|=q^{(d+1)n}$. Therefore, $[U:A]=q^{nd}$.
\end{proof}
%===========================================================================
\section{Proof of Theorem~\ref{Main-Theorem}}
In what follows the group of units of a given ring $R$ is denoted by $R^\times$. 
Let $(\rho, V)$ be a faithful representation of $\G\left(R\right)$. Let $\sigma:=\res{\rho}{U}$, be the restriction of $\rho$ to the Heisenberg subgroup $U$, and let $(\sigma_i,V_i)$, $1\leq i\leq k$, be the irreducible factors in the decomposition of the $U$-representation $(\sigma,V)$ with central characters $\chi_i$.  By Proposition~\ref{generic-prop}, we can assume that $\chi_1$ is a generic character of $Z(U)$, such that the element $b_1\in\mathfrak{p}^{-(n+\ell)}$ associated to $\chi_1$ in
Lemma \ref{Character-local}
satisfies $\nu(b_1)=-(n+\ell)$.  Then, by Proposition~\ref{polar-prop}, $A$ is a polarizing subgroup with respect to the generic character $\chi_1$. Therefore, by Stone-von Neumann theorem and Lemma~\ref{index}, we have 
$$
\dim(V_1)= [U:A]=q^{dn}.
$$
For any $\bar{\lambda}=\lambda+\mathfrak{p}^n\in R^\times$, and any root $\alpha\in\Phi$, consider the element $h_\alpha(\bar{\lambda})$ introduced in~\eqref{hlambda}. It follows from~\eqref{action-h} that $h_\alpha(\bar{\lambda})$ normalizes any one parameter subgroup. Therefore, $h_\alpha(\bar{\lambda})$ normalizes $U$.
Define the $U$-representation $\sigma^{\bar{\lambda},\alpha}$ to be the conjugation of $\sigma$ by $h_\alpha(\bar{\lambda})$:
\begin{equation*}
\sigma^{\bar{\lambda},\alpha}: U \to \GL(V),\qquad
 u \mapsto \sigma\left(h_\alpha(\bar{\lambda})uh_\alpha(\bar{\lambda})^{-1}\right).
\end{equation*}
Notice that the $U$-intertwining operator $\rho(h_{\alpha}(\bar{\lambda}))$ gives a $U$-isomorphism between $(\sigma,V)$ and $(\sigma^{\bar{\lambda},\alpha},V)$. Therefore, $(\sigma_1^{\bar{\lambda},\alpha},V_1)$ is also an irreducible subrepresentation of $(\sigma,V)$. Hence for any $z\in X_{\tilde\beta}\subseteq\CG(U)$ and $v\in V_1$ we have
$$
\chi_1(h_\alpha(\bar{\lambda}) zh_\alpha(\bar{\lambda})^{-1})v=\sigma_1(h_\alpha(\bar{\lambda}) zh_\alpha(\bar{\lambda})^{-1})v=\sigma_1^{\bar{\lambda},\alpha}(z)v=\chi_1^{\bar{\lambda},\alpha}(z)v,
$$
where $\chi_1^{\bar{\lambda},\alpha}$ is the one-dimensional representation of $Z(U)$ which is the central character of $\sigma_1^{\bar{\lambda},\alpha}$. Then for $s+\mathfrak{p}^n\in R\cong X_{\tilde{\beta}}$, from~\eqref{action-h}, we obtain
$$
\chi_1^{\bar{\lambda},\alpha}
\big(\bar{x}_{\tilde{\beta}}(s+\mathfrak{p}^n)\big)
=\chi_1\left(\bar{x}_{\tilde{\beta}}(\lambda^{\langle\tilde{\beta},\alpha\rangle}s+\mathfrak{p}^n)\right)=\psi\left(\Tr\left(b_1\lambda^{\langle\tilde{\beta},\alpha\rangle}s\right)\right).
$$
Next, we count the number of mutually distinct one-dimensional representations of $\chi_1^{\bar{\lambda},\alpha}$. We consider two cases:
\begin{enumerate}
\item The root system $\Phi$ is not
${\sf A}_1$ or ${\sf C}_m$ for $m\geq 2$: Equation~\eqref{F(Phi)} implies that there exists a root $\alpha\in\Phi$ such that $\langle \tilde{\beta},\alpha\rangle=1$. Hence, for this particular root $\alpha$ we have
$$
\chi_1^{\bar{\lambda},\alpha}
(\bar{x}_{\tilde{\beta}}(s+\mathfrak{p}^n))=
\chi_1(\lambda \bar{x}_{\tilde{\beta}}(s+\mathfrak{p}^n))=\psi\left(\Tr\left(b_1\lambda s\right)\right),\qquad \forall s\in \mathcal{O}.
$$
Since $\chi_1$ is corresponding to $b_1\in\mathfrak{p}^{-(n+\ell)}$ with $\nu(b_1)=-(n+\ell)$, we can conclude that $\chi_1^{\bar{\lambda}_1,\alpha}\neq \chi_1^{\bar{\lambda}_2,\alpha}$ when $\bar{\lambda}_1\neq \bar{\lambda}_2\in R^\times$, since otherwise $b_1(\lambda_1-\lambda_2)\in\mathfrak{p}^{-\ell}$ which implies $\lambda_1-\lambda_2\in\mathfrak{p}^n$. Hence, there are $q^n-q^{n-1}$ distinct one-dimensional representations $\chi_1^{\bar{\lambda},\alpha}$ and therefore, there are at least $q^n-q^{n-1}$ non-isomorphic irreducible subrepresentations of $V$, each of dimension $q^{dn}$. Hence,
$$
\dim(V)\geq(q^n-q^{n-1})q^{dn},
$$
where $d$ is given in Table~\ref{dvalue}.
\item The root system $\Phi$ is either ${\sf C}_m$ or ${\sf A}_1$: Equation~\eqref{F(Phi)} implies that for no root in $\Phi$, $\langle \tilde{\beta},\alpha\rangle=1$; but a root $\alpha$ can be chosen such that $\langle \tilde{\beta},\alpha\rangle=2$. Then for $s+\mathfrak{p}^n\in R\cong X_{\tilde{\beta}}$ we have
$$
\chi_1^{\bar{\lambda},\alpha}(
\bar{x}_{\tilde{\beta}}(s+\mathfrak{p}^n))
=\chi_1\left(
\bar{x}_{\tilde{\beta}}(
\lambda^2s+\mathfrak{p}^n)\right)=\psi\left(\Tr\left(b_1\lambda^2 s\right)\right).
$$
Hence, we can  construct $|R^\times|/2=(q^{n}-q^{n-1})/2$ distinct one-dimensional representations $\chi_1^{\bar{\lambda},\alpha}$, which leads to obtaining $(q^{n}-q^{n-1})/2$ non-isomorphic factors in the decomposition of $(\sigma,V)$, each of dimension $q^{(m-1)n}$.
Hence
$$
\dim(V)\geq \frac{1}{2}(q^{n}-q^{n-1})q^{(m-1)n}.
$$
\end{enumerate}
%============================================
\section*{Acknowledgement}
M.B. would like to thank Alexander Gamburd for his interest in this project.  During the completion of this work, M.B. was supported by a postdoctoral fellowship from the University of Ottawa. He wishes to thank his supervisor Vadim Kaimanovich for useful discussions on the subject of this paper and Dmitry Jakobson and Mikael Pichot for their support and stimulating discussions during his visit at McGill University where part of this joint work was done. 
C.K. and H.S. thank Monica Nevins and Kirill Zaynullin for interesting discussions concerning this project. The authors are also very grateful to the referee for a very careful reading and many helpful suggestions and corrections.

%===============================================================
\bibliographystyle{abbrv}

%\bibliographystyle{plain}
%\bibliography{faithful}
%=====================================
\end{document}